\documentclass[oneside]{amsart}
\usepackage[utf8]{inputenc}
\usepackage{amssymb,graphicx,hyperref}
\usepackage[scale=0.7]{geometry}

\newtheorem{lemma}{Lemma}
\newtheorem{theorem}[lemma]{Theorem}

\newtheorem{conjecture}[lemma]{Conjecture}
\newtheorem{prob}[lemma]{Problem}

\newcommand{\R}{\mathbb{R}}
\newcommand{\F}{\mathcal{F}}

\newcommand{\K}{\mathcal{K}}

\title{On a colorful problem by Dol'nikov concerning translates of convex bodies}

\author[L. Martínez-Sandoval]{Leonardo Martínez-Sandoval}
\address[L. Martínez-Sandoval]{Facultad de Ciencias, UNAM, Ciudad de México, México}
\email{leomtz@ciencias.unam.mx}

\author[E. Roldán-Pensado]{Edgardo Roldán-Pensado}
\address[E. Roldán-Pensado]{Centro de Ciencias Matemáticas, UNAM Campus Morelia, Morelia, Mexico}
\email{e.roldan@im.unam.mx}


\begin{document}
	
	\begin{abstract}
		In this note we study a conjecture by Jerónimo-Castro, Magazinov and Soberón which generalized a question posed by Dol'nikov. Let $\F_1,\F_2,\dots,\F_n$ be families of translates of a convex compact set $\K$ in the plane so that each two sets from distinct families intersect. We show that, for some $j$, $\bigcup_{i\neq j}\F_i$ can be pierced by at most $4$ points. To do so, we use previous ideas from Gomez-Navarro and Roldán-Pensado together with an approximation result closely tied to the Banach-Mazur distance to the square.
	\end{abstract}
	
	\maketitle
	
	\section{Introduction}
	
	In 2011 Dol'nikov posed the following problem \cite[Problem 8]{MRB2012}.
	\begin{prob}\label{prob}
		Let $\F_1$, $\F_2$ and $\F_3$ be families of translates of a convex compact set $\K$ in the plane such that $A\cap B\neq\emptyset$ for each $A\in \F_i$, $B\in \F_j$ with $i\neq j$. Is it always true that some $\F_i$ has piercing number at most $3$?
	\end{prob}
	The answer to this problem seems to be affirmative. The uncolored version (when $\F_1=\F_2=\F_3$) was solved affirmatively by Karasev \cite{Kar2000}, who later generalized it to higher dimensions \cite{Kar2008}.
	Jerónimo-Castro, Magazinov and Soberón \cite{JMS2015} gave a positive answer to Problem \ref{prob} when $\K$ is either centrally symmetric or a triangle. They also stated the following stronger conjecture.
	\begin{conjecture}\label{conj}
		For $n\ge 2$, let $\F_1,\F_2,\dots,\F_n$ be families of translates of a convex compact set $\K$ in the plane such that $A\cap B\neq\emptyset$ for each $A\in \F_i$, $B\in \F_j$ with $i\neq j$. Then there is some index $j$ such that $\bigcup_{i\neq j}\F_i$ has piercing number at most $3$.
	\end{conjecture}
	In the same paper they showed that this conjecture is true when $\K$ is an Euclidean disk.
	
	Recently Gomez-Navarro and Roldán-Pensado proved that Problem \ref{prob} has a positive answer when $\K$ is either of constant width or is close to a Euclidean disk with respect to the Banach-Mazur distance \cite{GR2023}. They also showed that Dol'nikov's problem has a positive answer with $8$ piercing points instead of $3$ and that Conjecture \ref{conj} is true with $9$ piercing points instead of $3$.
	
	The purpose of this paper is to prove Conjecture \ref{conj} with $4$ piercing points instead of $3$.
	\begin{theorem}\label{thm}
		For $n\ge 2$, let $\F_1,\F_2,\dots,\F_n$ be families of translates of a convex compact set $\K$ in the plane such that $A\cap B\neq\emptyset$ for each $A\in \F_i$, $B\in \F_j$ with $i\neq j$. Then there is some index $j$ such that $\bigcup_{i\neq j}\F_i$ has piercing number at most $4$.
	\end{theorem}
	
	The proof follows the ideas used to prove Theorem 2.3 from \cite{GR2023}, together with an approximation result which is related to the Banach-Mazur distance to the square. The auxiliary results we require are stated in Section \ref{sec:lemmas}. Section \ref{sec:proof} contains the proof of Theorem \ref{thm}.
	
	\section{Two auxiliary lemmas}\label{sec:lemmas}
	
	Our proof is based on two lemmas. The first one is a special case of a theorem proved by Gomez-Navarro and Roldán-Pensado \cite[Theorem 2.5(a)]{GR2023}.
	\begin{lemma}\label{lem:transversal}
		For $n\ge 2$, let $\F_1,\F_2,\dots,\F_n$ be families of translates of a convex compact set $\K$ in the plane such that $A\cap B\neq\emptyset$ for each $A\in \F_i$, $B\in \F_j$ with $i\neq j$. If for every index $j$ the family $\bigcup_{i\neq j}\F_i$ has piercing number larger than $3$, then there is a line transversal to $\bigcup_i \F_i$.
	\end{lemma}
	
	The tools behind the proof of Lemma \ref{lem:transversal} are the uncolored version of Problem \ref{prob} and the fact that a family of convex sets $\F$ on the plane has a transversal line in every direction if and only if $\F$ is pairwise intersecting. We refer the reader to \cite{GR2023} for the full details.
	
	The second lemma essentially gives a way of approximating a convex body by a parallelogram. It implies that the Banach-Mazur distance from the square to any planar convex body is at most $2$. This was already known \cite[Theorem 5.5]{GLMP2004}, however we require something slightly stronger. Given a convex body $\K$, it is known that the parallelogram $P$ of maximal area contained in $\K$ satisfies that there is a translation of $2P$ that contains $\K$. We require a similar result where, instead of $P$ having maximal area, the direction of one of the sides of $P$ is fixed.
	\begin{lemma}\label{lem:parallelogram}
		Let $\K$ be a convex body in the plane and let $u$ be a fixed direction. Then, there is a parallelogram $P\subset \K$ such that one of the sides of $P$ has direction $u$ and there is a translated copy $Q$ of $2P$ such that $\K\subset Q$.
	\end{lemma}
	\begin{proof}
		We may assume that $\K$ is smooth, as the general case follows from standard approximation arguments. Without loss of generality, the direction $u$ is horizontal and the bottom and top horizontal supporting lines of $\K$ are $y=0$ and $y=1$, respectively.
		
		\begin{figure}
			\includegraphics{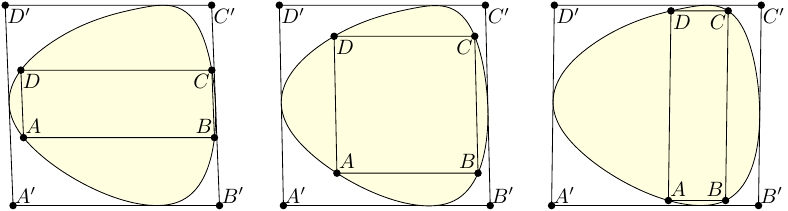}
			\caption{The parallelograms vary the ratio between their sides.}
			\label{fig:parallelograms}
		\end{figure}
		
		The length $l(h)$ of the horizontal chord of $\K$ at height $h\in [0,1]$ depends continuously on $h$ and it is unimodular: $l(0)=0$, then it increases until it attains some maximum $m$ and then goes back to $0$. Therefore, each $l\in[0,m)$ is attained exactly twice.
		
		For every $l\in (0,m)$, let $ABCD$ be the inscribed parallelogram to $\K$ such that $AB$ and $CD$ are horizontal, and $AB=CD=l$. Let $A'B'C'D'$ be the parallelogram circumscribed around $\K$ such that the sides of $A'B'C'D'$ are parallel to the sides of $ABCD$.  See Figure \ref{fig:parallelograms}.
		
		Note that $A'B'C'D'$ is homothetic to $ABCD$ if and only if $A'B'/B'C' = AB/BC$. Let $\alpha$ be the interior angle $\angle DAB$ and assume that $0<r<R$ are real numbers such that there is a disk of radius $r$ contained in $\K$ and $\K$ is contained in a disk of radius $R$. Then $B'C'=1/\sin(\alpha)$ and $2r/\sin(\alpha)\le A'B' \le 2R/\sin(\alpha)$, therefore $2r\le A'B'/B'C'<2R$. However $AB/BC$ tends to $0$ (resp.  $\infty$) as $l$ tends to $0$ (resp. $m$). This variation is continuous, so at some point $ABCD$ and $A'B'C'D'$ are homothetic.
		
		By applying a linear transformation, we may assume that these parallelograms are squares and that $ABCD$ has unit side. Our result follows immediately if we manage to prove the following claim: if $P$ is a unit square and $P'$ is a square homothetic to $P$ such that $P\subset \K\subset P'$, the vertices of $P$ are in $\partial \K$ and $\K$ is internally tangent to the sides of $P'$, then the homothety ratio between $P$ and $P'$ is at most $2$. If this is the case, then $P\subset \K$, and for a translation $Q$ of $2P$ we would have $\K\subset P' \subset Q$ as desired. 
		
		\begin{figure}
			\includegraphics{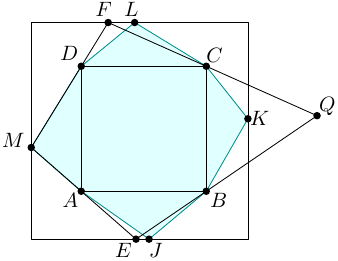}
			\caption{The points where $\K$ touches $P$ and $P'$, and the point $Q$.}
			\label{fig:squares}
		\end{figure}
		
		Recall that the vertices $A,B,C,D$ of $P$ lie in $\partial \K$ and let $J,K,L,M$ be the points where the sides of $P'$ touch $\partial \K$, as in Figure \ref{fig:rightpoint}. These eight points form a convex polygon. Let $E$ be the intersection of the line $MA$ with the bottom side of $P'$ and let $F$ be the intersection of the $MD$ with the top side of $P'$.
		
		By convexity at the angle $LDM$, we have that $F$ lies to the left of $L$. In turn, by convexity at the angle $KCL$ we have that $L$ lies to the left of line $CB$. Therefore, $F$ (and analogously $E$) lies to the left of line $CB$. Then $K$ lies below the line $FC$ and above the line $EB$. We conclude that $K$ lies to the left of the intersection $Q$ of the lines $FC$ and $EB$.
		
		Proceeding by contradiction, we show that if the desired homothety ratio is larger than $2$, then $Q$ is strictly to the left of the side $B'C'$ of $P'$, which is impossible since by the previous argument then $K$ would not be on the side $B'C'$ of $P'$.
		
		As in Figure \ref{fig:rightpoint}, let $\ell$ be the horizontal line through $M$. Define $a$ as the distance from $D$ to $\ell$ and $h$ as the distance from $L$ to $\ell$. Set $X$ and $Y$ as the intersections of the lines $FC$ and $EB$ with $\ell$, respectively. Let $X'$ be the intersection of the line $AB$ and the vertical line through $X$ and define $Y'$ as the intersection of the line $CD$ and the vertical line through $Y$.
		
		\begin{figure}
			\includegraphics{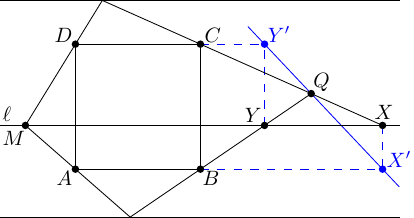}
			\caption{Auxiliary points and lines.}
			\label{fig:rightpoint}
		\end{figure}
		
		Recall that $P$ is a unit square, so $MX=MX/CD=h/(h-a)$. This shows that $X$ depends on the vertical position of $P$ but not on the horizontal position of $P$. The same is true for $Y$ and consequently it is also true for $X'$ and $Y'$. At this point, we may ignore the specific convex body $\K$ and study the possible diagrams we may obtain.
		
		We claim that, if we translate $P$ horizontally (i.e. in diagrams where $P$ has the same vertical position), $Q$ stays on the line $X'Y'$. Indeed, the lines $BX'$, $CY'$ and $XY$ are horizontal, and thus projectively concurrent. The same is true for the lines $BC$, $YY'$ and $XX'$, since they are vertical. Therefore, by the dual of Pappus' theorem \cite{Cox1961} the lines $BY$, $CX$ and $X'Y'$ are concurrent, which means that $Q$ lies on the line $X'Y'$.
		
		Let us assume, that $Q$ is above the line $\ell$. Then $X$ is to the right of $Q$ and $Y$ is to the left of $Q$. This implies that $X'$ is below and to the right of $Y'$. As $P$ is translated to the left, the point $F$ also moves to the left and therefore the line $XF$ intersects the line $X'Y'$ at a lower point. Since this point is $Q$ and the slope of $X'Y'$ is not positive then $Q$ moves to the right as $P$ is translated to the left (i.e. among all possible diagrams where $P$ has the same vertical position, the one where $P$ is leftmost has rightmost $Q$). If $Q$ is below $\ell$ the reasoning is analogous.
		
		If $Q$ lies on $\ell$ then $X'Y'$ is vertical and $Q$ does not move horizontally as $P$ is translated left.
		
		Hence, we only need to prove that $Q$ lies strictly to the left of the side $B'C'$ of $P'$ in the limit case when the square $P$ has its left side contained in the left side of $P'$, as in Figure \ref{fig:extremesquare}.
		
		\begin{figure}
			\includegraphics{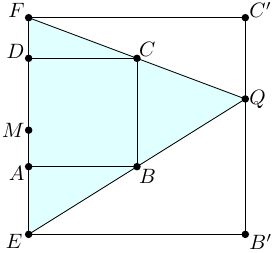}
			\caption{The extreme case when the homothety ratio is exactly $2$.}
			\label{fig:extremesquare}
		\end{figure}	
		
		In this case, if the homothety ratio is exactly $2$, then $CQB$ and $FQE$ are homothetic triangles from $Q$, so the distance from $Q$ to $CB$ is half the distance from $Q$ to $FE$. This means that $Q$ lies exactly on the side $B'C'$ of $P'$. Then, if the homothety ratio is larger than $2$, then the distance from $Q$ to $CB$ is less than half the distance from $Q$ to $FE$, and thus $Q$ is strictly to the left of side $B'C'$. This implies that $K$ is also strictly to the left of the side $B'C'$ of $P'$, which is the desired contradiction to our original assumption that the homothety ratio was larger than $2$. We conclude that the homothety ratio of both squares is at most $2$, as desired.
	\end{proof}
	
	\section{Proof of Theorem \ref{thm}} \label{sec:proof}
	
	Once we have the Lemmas from the previous section, the proof of Theorem \ref{thm} is simple.
	We start by using Lemma \ref{lem:transversal}. If for some index $j$ it happens that $\bigcup_{i\neq j} \F_i$ has piercing number at most $3$, then we are done. Otherwise, there is a line $\ell$ transversal to $\bigcup_i \F_i$ with direction, say, $u$. By Lemma \ref{lem:parallelogram}, there is a paralellogram $P\subset \K$ such that one of the sides of $P$ has direction $u$ and there is a translated copy $R$ of $2P$ such that $\K\subset R$. Let $v$ be the direction of the other side of the paralellogram $P$.
	
	By projecting the convex bodies in the sets $\F_i$ to a line $m$ orthogonal to $v$, we obtain collections $\mathcal{I}_1,\ldots, \mathcal{I}_n$ of intervals on $m$ such that $I \cap I'\neq \emptyset$ for each $I\in \mathcal{I}_i$, $I'\in \mathcal{I}_j$ with $i\neq j$. What follows is a common generalization of the colorful Helly theorem (see e.g. \cite{Bar2021}). If every two intervals in $\bigcup_i \mathcal{I}_i$ intersect, then by Helly's theorem there is a point common to all of them. If not, there are two of these intervals, say $I$ and $I'$, that are disjoint. These intervals must then belong to the same family $I_j$ and therefore any interval $I^\ast$ not in this family must intersect both $I$ and $I'$. Thus, $I^\ast$ contains any point separating $I$ from $I'$. In both cases there is an index $j$ such that the elements of $\bigcup_{i\neq j} \mathcal{I}_i$ have a point in common. By lifting this point in direction $v$, we obtain a line transversal $\ell'$ to $\bigcup_{i\neq j} \F_i$ with direction $v$. This implies that every translate of $\K$ in $\bigcup_{i\neq j} \F_i$ intersects both the line $\ell$ with direction $u$ and the line $\ell'$ with direction $v$. We now exhibit four points that pierce all translates of $\K$ with this property.
	
	Consider the sets
	\begin{align*}
		\mathbf{K}&=\{x\in\R^2: \text{$\K+x$ intersects both $\ell$ and $\ell'$}\}\text{ and}\\
		\mathbf{R}&=\{x\in\R^2: \text{$R+x$ intersects both $\ell$ and $\ell'$}\}.
	\end{align*}
	Note that the set $\mathbf{R}$ is congruent to $R$ and, since $P$ is a parallelogram, the set $-P$ is congruent to $P$. Hence, the set $\mathbf{R}$ can be covered with four copies of $-P$, say $-P+a$, $-P+b$, $-P+c$ and $-P+d$. Then the points $a$, $b$, $c$ and $d$ pierce any translate $\K+x$ that intersect both $\ell$ and $\ell'$. Indeed, if $x\in \mathbf{K}\subset \mathbf{R}$ then $x$ belongs to either $-P+a$, $-P+b$, $-P+c$ or $-P+d$. Without loss of generality we may assume that $x\in-P+a$ which implies that $a\in \K+x$. \qed
	
	\section{Acknowledgments}
	We would like to thank two anonymous referees whose comments helped to improve the presentation of this note. This work was supported by UNAM-PAPIIT project IN111923.
	
	\bibliographystyle{amsplain}
	\bibliography{main}

\providecommand{\bysame}{\leavevmode\hbox to3em{\hrulefill}\thinspace}
\providecommand{\MR}{\relax\ifhmode\unskip\space\fi MR }
\providecommand{\MRhref}[2]{%
  \href{http://www.ams.org/mathscinet-getitem?mr=#1}{#2}
}
\providecommand{\href}[2]{#2}
\begin{thebibliography}{GLMP04}

\bibitem[B{\'a}r21]{Bar2021}
I.~B{\'a}r{\'a}ny, \emph{Combinatorial convexity}, University Lecture Series,
  vol.~77, American Mathematical Society, Providence, RI, 2021.

\bibitem[Cox61]{Cox1961}
H.~S.~M. Coxeter, \emph{Introduction to geometry}, John Wiley \& Sons, Inc.,
  New York-London, 1961.

\bibitem[GLMP04]{GLMP2004}
Y.~Gordon, A.~E. Litvak, M.~Meyer, and A.~Pajor, \emph{John's decomposition in
  the general case and applications}, Journal of Differential Geometry
  \textbf{68} (2004), no.~1, 99--119.

\bibitem[GNRP23]{GR2023}
C.~Gomez-Navarro and E.~Rold{\'a}n-Pensado, \emph{Transversals to colorful
  intersecting convex sets}, arXiv preprint arXiv:2305.16760 (2023), 1--14.

\bibitem[JCMS15]{JMS2015}
J.~Jer{\'o}nimo-Castro, A.~Magazinov, and P.~Sober{\'o}n, \emph{On a problem by
  {D}ol’nikov}, Discrete Mathematics \textbf{338} (2015), no.~9, 1577--1585.

\bibitem[Kar00]{Kar2000}
R.~N. Karasev, \emph{Transversals for families of translates of a
  two-dimensional convex compact set}, Discrete \& Computational Geometry
  \textbf{24} (2000), 345--354.

\bibitem[Kar08]{Kar2008}
\bysame, \emph{Piercing families of convex sets with the d-intersection
  property in {R}d}, Discrete \& Computational Geometry \textbf{39} (2008),
  no.~4, 766--777.

\bibitem[MRB12]{MRB2012}
J.~Matou{\v{s}}ek, G.~Rote, and I.~B{\'a}r{\'a}ny, \emph{Discrete {G}eometry},
  Oberwolfach Reports \textbf{8} (2012), no.~3, 2459--2548.

\end{thebibliography}
\end{document}